\documentclass[11pt,reqno]{amsart}
\usepackage{amssymb,latexsym}
\usepackage{graphicx}
\usepackage{fancyhdr,amssymb,tikz,epstopdf}
\usepackage{url}		

\newcommand{\N}{{\mathbb N}}
\newcommand{\Z}{{\mathbb Z}}

\theoremstyle{plain}
\numberwithin{equation}{section}
\newtheorem{thm}{Theorem}[section]
\newtheorem{theorem}[thm]{Theorem}
\newtheorem{lemma}[thm]{Lemma}

\begin{document}
\fancyhead{}
\renewcommand{\headrulewidth}{0pt}

\setcounter{page}{1}

\title[A graph-theoretic encoding of Lucas sequences]{A graph-theoretic encoding of Lucas sequences}
\author{James Alexander}
\address{James Alexander \\
	Department of Mathematical Sciences\\
                University of Delaware\\
                Newark, Delaware\\
                United States}
\email{alex@math.udel.edu}
\author{Paul Hearding}
\address{Paul Hearding\\
	Department of Mathematical Sciences\\
                University of Delaware\\
                Newark, Delaware\\
                United States}
\email{hearding@math.udel.edu}
\begin{abstract}
Some well-known results of Prodinger and Tichy are that the number of independent sets in the $n$-vertex path graph is $F_{n+2}$, and that the number of independent sets in the $n$-vertex cycle graph is $L_n$. We generalize these results by introducing new classes of graphs whose independent set structures encode the Lucas sequences of both the first and second kind. We then use this class of graphs to provide new combinatorial interpretations of the terms of Dickson polynomials of the first and second kind.
\end{abstract}

\maketitle

\section{Introduction and main results}
 For any graph $G$, we call a set $S$ of vertices of $G$ an \emph{independent set} if no two vertices of $S$ are adjacent. We let $i(G)$ denote the total number of independent sets of $G$ and, for each $t\in\N$, we let $i_t(G)$ denote the number of independent sets of $G$ of size $t$; thus, $i(G)=\sum_{t\geq 0}i_t(G)$. The quantity $i(G)$ was first explicitly considered by Prodinger and Tichy in \cite{PRODTICH}, who referred to it as the Fibonacci number of a graph. We present  two of their main results as the following theorem. Here, $P_n$ denotes the $n$-vertex path graph, $C_n$ denotes the $n$-vertex cycle graph (where the $1$-vertex cycle is taken to be a vertex with a loop, and the $2$-vertex cycle is taken to be a single edge), and we adopt the common conventions $F_0=0$, $F_1=1$, $L_0=2$, and $L_1=1$. 
 
 \begin{theorem}\label{prod}
For any $n\in\N$, 
\begin{equation}
i(P_n)=F_{n+2},
\end{equation}•
and 
\begin{equation}
i(C_n)=L_n. 
\end{equation}•
\end{theorem}•

Our main result will be a generalization of Theorem~\ref{prod}. For this, we will define two new classes of graphs. Fix any $n,a,b\in\N$ with $a\geq b$. Create an $n$-vertex cycle with vertex set $\Z_n$; for each vertex of the cycle, create an $a$-vertex complete graph sharing with the cycle only this vertex. Then, for each $v\in\Z_n$, make vertex $v$ adjacent to $a-b$ additional vertices of the complete graph containing vertex $v+1$ (mod $n$), and denote this graph $C(n,a,b)$. For example, $C(6,5,3)$ is the graph:
\begin{figure}[h!]
\includegraphics[width=1.5in]{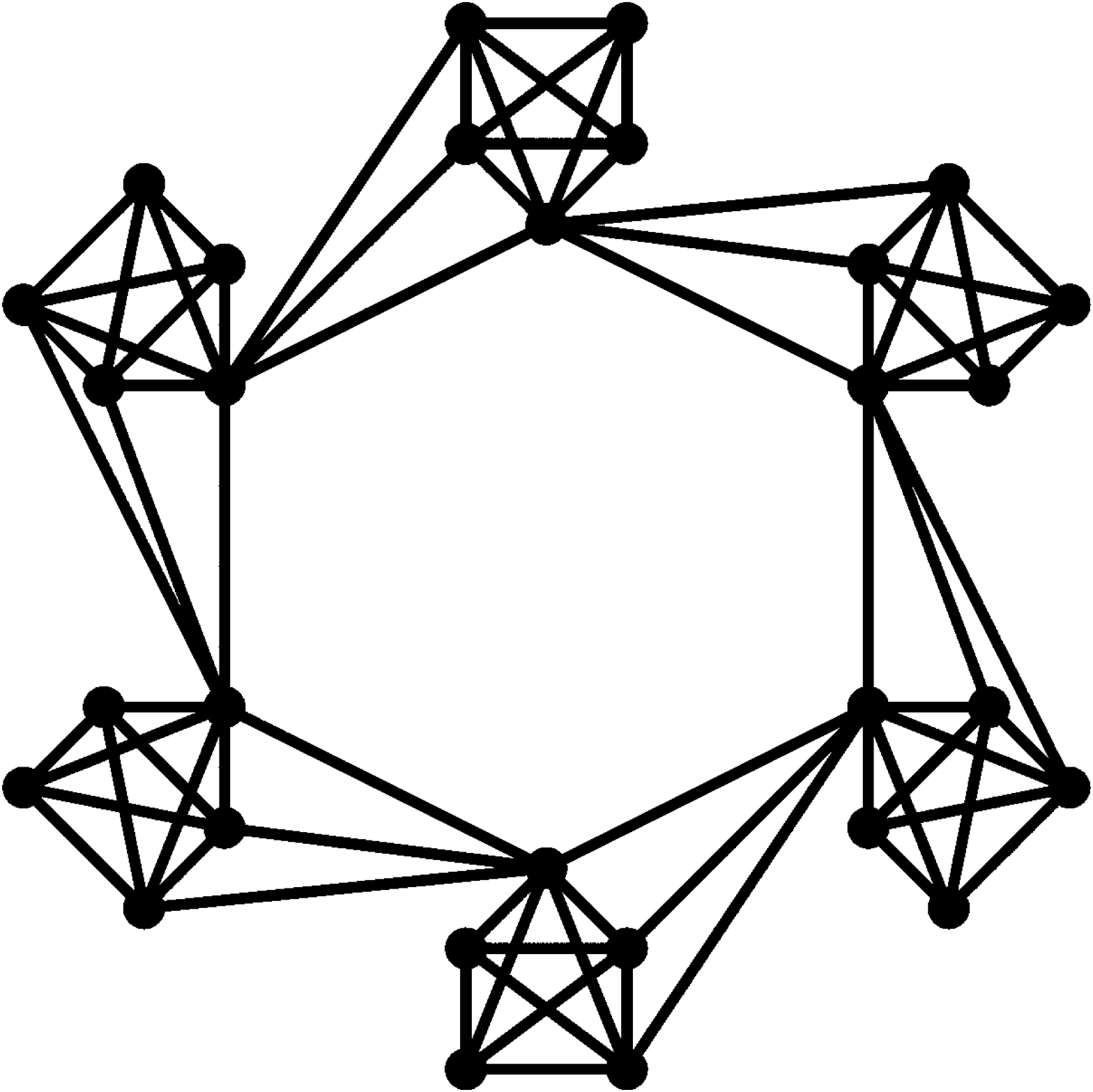}.
\end{figure}

\newpage

We refer to this class of graphs, over all valid $n,a,b\in\N$, as \emph{chainsaw graphs}. When referring to a particular chainsaw graph $C(n,a,b)$, we call the $n$ vertices lying on the inner cycle its \emph{chain vertices}, and we call the set of remaining vertices its \emph{blade vertices}. This will serve as our generalization of $C_n$, as we will soon see. We generalize the path graph to a graph which we denote $P(n,a,b)$ by considering $C(n+1,a,b)$ and removing one of the chain vertices (e.g., vertex $0$) and all edges adjacent to it. We call these graphs \emph{broken chainsaws}, and refer to the vertices similarly as chain and blade vertices. 

With these definitions in place, we now state our generalization of Theorem~\ref{prod}. As is common, we let $U_n(a,b)$ and $V_n(a,b)$ denote the Lucas sequences of the first and second kind, respectively. That is, we let $U_0(a,b)=0$, $U_1(a,b)=1$, and $U_n(a,b)=aU_{n-1}(a,b)-bU_{n-2}(a,b)$ for $n>1$ (so that $U_n(1,-1)$ are the Fibonacci numbers);  we let $V_0(a,b)=2$, $V_1(a,b)=a$, and $V_n(a,b)=aV_{n-1}(a,b)-bV_{n-2}(a,b)$ for $n>1$ (so that $V_n(1,-1)$ are the Lucas numbers).

\begin{theorem}\label{chainsaw}
For any $n,a,b\in\N$ satisfying $a\geq b$, we have that 
\begin{equation}
i(P(n,a,b))=U_{n+2}(a,-b),\label{pathgen}
\end{equation}• 
and that
\begin{equation}
i(C(n,a,b))=V_n(a,-b). \label{cyclegen}
\end{equation}•
\end{theorem}•

We prove this Theorem in Section~\ref{proofs} while discussing some relationships between Dickson polynomials and Lucas sequences and providing some graph-theoretic interpretations of these well-studied objects.  We note that Theorem~\ref{prod} is the special case of Theorem~\ref{chainsaw} when $a=b=1$.

\section{Relationships to Dickson polynomials and a proof of Theorem~\ref{chainsaw}}\label{proofs}

In this section we examine the relationship between Dickson Polynomials and Lucas sequences and discuss some results which will be crucial to proving Theorem~\ref{chainsaw}. In the process, we provide new graph-theoretic interpretations of Lucas sequences and Dickson polynomials. As is common, we use $D_n(X,Y)$ and $E_n(X,Y)$ to denote Dickson polynomials of the first and second kind, respectively. That is, we let 
\begin{equation}
D_n(X,Y):=\sum_{t=0}^{\lfloor n/2\rfloor}\frac{n}{n-t}\binom{n-t}{t}(-Y)^tX^{n-2t},\label{firstkind}
\end{equation}•
and
\begin{equation}
E_n(X,Y):=\sum_{t=0}^{\lfloor n/2\rfloor}\binom{n-t}{t}(-Y)^tX^{n-2t}.\label{secondkind}
\end{equation}•
We start with the following result which is known in finite field theory. See, for example, \cite{KIMBO}, \cite{KIMBT}, or \cite{DPAP} for more on this result.  For more information on Dickson polynomials in general, see \cite{DPAP}.
\begin{theorem}
For any $n\in\N$ and $a,b\in\Z$, 
\begin{equation}
D_n(a,b)=V_n(a,-b),
\end{equation}•
and 
\begin{equation}
E_n(a,b)=U_{n+1}(a,-b). 
\end{equation}•
\end{theorem}•

We will prove Theorem~\ref{chainsaw} by showing that the $t^{\textnormal{th}}$ term of (\ref{firstkind}) and the $t^{\textnormal{th}}$  term of (\ref{secondkind}) can be graph-theoretically interpreted as the number of independent sets in the chainsaw graph $C(n,a,b)$ and the broken chainsaw graph $P(n,a,b)$, respectively, which contain exactly $t$ chain vertices. For this, we will need the following result, which is well known in graph theory, and is not difficult to prove. See, for example, \cite{FIXIND}. 

\begin{lemma}
For any $n\in\N$ and $t\in\N_0$, we have 
\begin{equation}
i_t(P_n)=\binom{n-t+1}{t},\label{oldstuffpath}
\end{equation}•
and we also have
\begin{equation}
i_t(C_n)=\frac{n}{n-t}\binom{n-t}{t}. \label{oldstuff}
\end{equation}•
\end{lemma}•

With this in place, we are now ready to proceed to the proof of Theorem~\ref{chainsaw}. 

\begin{proof}[Proof of Theorem~\ref{chainsaw}]
Fix $n,a,b\in\N$ so that $a\geq b$. As previously discussed, it follows from  (\ref{secondkind}) and (\ref{oldstuffpath}) that (\ref{pathgen}) holds if the number of independent sets in $P(n,a,b)$ which contain exactly $t$ chain vertices for $t\in\N_0$ is given by 
\begin{equation}
i_t(P_n)b^ta^{n-2t+1}.\label{firstob}
\end{equation}•
 First, note that the number of ways to choose $t$ independent chain vertices in $P(n,a,b)$, by definition, is $i_t(P_n)$. Then, once $t$ independent chain vertices are chosen, there are $t$ sets of $b-1$ blade vertices and $n-2t+1$ sets of $a-1$ blade vertices with which they share no adjacencies, so (\ref{firstob}) holds. A similar argument shows that the number of independent sets in $C(n,a,b)$ which contain exactly $t$ chain vertices for $t\in\N_0$ is given by 
 \begin{equation}
i_t(C_n)b^ta^{n-2t},
\end{equation}•
and thus, by (\ref{firstkind}) and (\ref{oldstuff}), we have (\ref{cyclegen}).
\end{proof}

A graph-theoretic interpretation of the Lucas Sequence is now established by Theorem~\ref{chainsaw}, and from its proof emerges a graph-theoretic interpretation of the terms of the Dickson polynomial. 

\section{Acknowledgments}

We would like to thank Professor Robert Coulter for bringing our
attention to the connection between Lucas Sequences and Dickson polynomials, which led to
a much shorter proof of our main result. We also thank both him and Professor Felix Lazebnik for their suggestions on the preliminary draft of this note. Finally, we thank the referee for carefully reviewing this note, and for catching a crucial indexing error.

\end{document}